\documentclass[12pt]{amsart}
   \usepackage{amsmath,amsthm}
   \usepackage{amsfonts}   
   \usepackage{amssymb}    
   \usepackage{url}

\advance \topmargin by -\headheight
\advance \topmargin by -\headsep
      
\evensidemargin \oddsidemargin
\usepackage[all]{xy}

\newtheorem{thm}{Theorem}[section]
\newtheorem{prop}[thm]{Proposition}

\newtheorem{cor}[thm]{Corollary}

\newtheorem{qu}[thm]{Question}

\theoremstyle{remark}
\newtheorem{rem}[thm]{Remark}
\newtheorem{ex}[thm]{Example}

\title{Exactly fillable contact structures without Stein fillings}
\author{Jonathan Bowden}
\address{Mathematisches Institut, Universit\"at Augsburg, Universit\"atsstr.\ 14, 86159 Augsburg, Germany}
\curraddr{Max-Planck-Institut f\"ur Mathematik, Vivatsgasse 7, 53129 Bonn, Germany}
\email{jonathan.bowden@math.uni-augsburg.de}
\date{\today}
\begin{document}

\maketitle

\begin{abstract}
We give examples of contact structures which admit exact symplectic fillings, but no Stein fillings, answering a question of Ghiggini.
\end{abstract}

\section{Introduction}
It is a fundamental problem is contact topology to determine which contact 3-manifolds admit symplectic fillings. There are many varieties of symplectic fillings. In addition to weak and strong fillings it is natural to consider contact manifolds that are exactly fillable and if the filling has a complex structure, then the natural class of symplectic fillings are those that are Stein. The relationship between these various notions is depicted in following sequence of inclusions:
\[ \{  \text{Stein fillable}  \} \subset \{  \text{Exactly fillable}  \} \subsetneq \{  \text{Strongly fillable}  \} \subsetneq \{  \text{Weakly fillable}  \} \subsetneq \{  \text{Tight}  \}. \]
Let us emphasise that a symplectic filling is always required to have connected boundary. Examples of strongly fillable contact structures that are not exactly fillable were found by Ghiggini, \cite{Ghi}. Examples of weakly fillable contact structures that admit no strong fillings were first discovered by Eliashberg, \cite{Eli3}. Finally, Etnyre and Honda, \cite{EtH} showed that there exist tight contact structures that are not weakly fillable. So all these inclusions are strict except possibly for the first. The main result of this paper is that the first inclusion is also strict.
\begin{thm}\label{Main_thm}
There exist exactly fillable contact structures that admit no Stein fillings.
\end{thm}
\noindent This answers a question raised by Ghiggini whilst studying the relationship between strong and Stein fillability (\cite{Ghi}, p.\ 1686). The contact structures of Theorem \ref{Main_thm} are obtained by using the fact that the Brieskorn spheres $\Sigma(2,3,6n+5)$ considered in \cite{Ghi} can be realised as coverings of Seifert fibred manifolds that are compact quotients of $PSL(2,\mathbb{R})$. One then constructs an exactly fillable contact structure on the connected sum $\overline{\Sigma}(2,3,6n+5) \# \Sigma(2,3,6n+5)$ using the $PSL(2,\mathbb{R})$-structure in an explicit way. By a result of Eliashberg a connected sum of contact manifolds is Stein fillable if and only if each of the summands is Stein fillable. However, the contact structure on $\overline{\Sigma}(2,3,6n+5)$ is Ghiggini's non-Stein fillable contact structure, which is in fact not even exactly fillable.

This then exhibits the failure of Eliashberg's result for exactly fillable contact structures and also implies that the notion of exact fillability is strictly weaker than that of exact semi-fillability, where one allows fillings with disconnected boundary. This latter fact is perhaps surprising since the notions of strong and weak fillability do not depend on whether one requires the boundaries of the fillings to be connected or not. Moreover, since the contact structures in question are perturbations of taut foliations, we further deduce that perturbations of taut foliations on homology spheres are not necessarily Stein fillable.

\subsection*{Acknowledgments:}
The author would like to thank T.\ Vogel and P.\ Ghiggini for helpful suggestions. He would also like to thank K.\ Cieliebak for pointing an error in an earlier version of this article and for providing an advanced version of \cite{CE}. The hospitality of the Max Planck Institut f\"ur Mathematik in Bonn, where part of this research was carried out, is gratefully acknowledged.

\subsection*{Conventions:}
All manifolds are smooth, oriented and connected and all contact structures will be assumed to be oriented and positive.

\section{Stein fillings of non-prime Manifolds}
In \cite{Eli2}, Eliashberg states a result about decomposing symplectic fillings of non-prime manifolds. For our purposes we will be content with the case of Stein fillings, for which a detailed proof can now be found in \cite{CE}.
\begin{thm}[\cite{Eli2}, Section 8, \cite{CE}, Theorem 16.7]\label{connect_Stein}
Let $X$ be a Stein filling with boundary $M = (M_1,\xi_1) \# (M_2,\xi_2)$ that decomposes as a non-trivial connected sum of contact manifolds. Then $X$ decomposes as a boundary connect sum $X = X_1 \#_{\partial} X_2$, where $X_1,X_2$ are Stein fillings of $(M_1,\xi_1)$ and $(M_2,\xi_2)$ respectively.
\end{thm}
The connected sum operation is well-defined on tight contact manifolds by \cite{Col}. Moreover, if both $(M_1,\xi_1)$ and $(M_2,\xi_2)$ are Stein fillable, then attaching a Stein 1-handle yields a Stein filling of the contact connected sum. In this way Theorem \ref{connect_Stein} implies the following as a corollary. 
\begin{cor}\label{connect_Stein_rem}
A connected sum of contact manifolds $(M_1,\xi_1) \# (M_2,\xi_2)$ is Stein fillable if and only if $(M_1,\xi_1)$ and  $(M_2,\xi_2)$ are Stein fillable.
\end{cor}


\section{Non-Stein Exact Symplectic Fillings}
We use a construction which goes back to McDuff, \cite{McD} to construct many examples of exact symplectic fillings $(X, d\lambda)$ with $H_3(X) \neq 0$. Since the third homology is non-trivial, these fillings, although exact, cannot be Stein. The starting point for the construction is an exact symplectic filling of the form $(M \times [0,1], d \lambda)$ both of whose ends are convex, which can, for example, be obtained by considering compact quotients of $PSL(2, \mathbb{R})$. This was first observed by Geiges and independently by Mitsumatsu.
\begin{ex}[\cite{Gei}, \cite{Mit2}]\label{psl_ex}
Let $\mathfrak{psl}_2$ denote the lie algebra of $PSL(2, \mathbb{R})$ and choose the following basis:
\[h = \frac{1}{2}\left( \begin{array}{cc}
1 & 0  \\
0 & -1  \\
 \end{array} \right), \ l = \frac{1}{2}\left( \begin{array}{cc}
0 & 1  \\
1 & 0  \\
 \end{array} \right),  \ k = \frac{1}{2}\left( \begin{array}{cc}
0 & -1  \\
1 & 0  \\
 \end{array} \right). \]
We identify $\mathfrak{psl}^*_2$ with the space of left-invariant 1-forms on $PSL(2, \mathbb{R})$ and define a linking pairing by
\[LK(\alpha,\beta) = \alpha \wedge d \beta.\]
With respect to this pairing the ordered basis $\{h^*,l^*,k^*\}$ is orthogonal and
\[LK(h^*,h^*) = LK(l^*,l^*) =  -1  \text{ and } LK(k^*,k^*) = 1.\]
Any non-zero 1-form $\alpha \in \mathfrak{psl}^*_2$ then defines a positive resp.\ negative contact structure or a taut foliation, depending on whether $LK(\alpha,\alpha)$ is positive resp.\ negative or zero. We let $\Gamma$ be a co-compact lattice in $PSL(2, \mathbb{R})$ and consider $M =  PSL(2, \mathbb{R})/ \Gamma$. If we set \[ \lambda = t k^* + (1-t)h^*\]on $M \times [0,1]$, then the pair $(M \times [0,1], d \lambda)$ is a symplectic filling with convex ends.
\end{ex}
Other examples of symplectic structures on $M \times [0,1]$ with convex ends are given by $T^2$-bundles over $S^1$ with Anosov monodromy or by smooth volume preserving Anosov flows (cf.\ \cite{Mit2}). It is now easy to construct examples of non-Stein exact fillings: one simply attaches a symplectic $1$-handle to $(M \times [0,1], d \lambda)$ with ends in each component of the boundary. 
\begin{prop}[\cite{McD}]\label{non_Stein_exact}
There exist exact, non-Stein symplectic fillings.
\end{prop}
\begin{proof}
Let $(M \times [0,1], d \lambda) = (X, \omega)$ be an exact symplectic filling with convex ends. Attach a symplectic $1$-handle to obtain a filling of the connected sum $(\overline{M}, \xi_0) \# (M, \xi_1)$. We denote this new filling by $\widetilde{X} = X \cup e_1$, where $e_1$ denotes a topological $1$-handle. The symplectic form on $\widetilde{X}$ restricts to $\omega$ on $X$. Thus by the long exact sequence in cohomology of the pair $(\widetilde{X}, X)$ we see that $\widetilde{X}$ is an exact filling and the hypersurface $M \times \frac{1}{2} \subset \widetilde{X}$ is non-separating, whence $H_3(\widetilde{X}) \neq 0$ and $\widetilde{X}$ cannot be Stein.
\end{proof}
The manifolds $(N,\xi) = (\overline{M}, \xi_0) \# (M,\xi_1)$ in Proposition \ref{non_Stein_exact} are always exactly fillable, but their natural fillings are not Stein. This raises the question of whether they are always Stein fillable or not. Or equivalently whether $(\overline{M}, \xi_0)$ and $(M,\xi_1)$ are always Stein fillable. We will answer this question, by considering various Brieskorn spheres, which can be realised as finite covers of compact quotients of $PSL(2,\mathbb{R})$. 

Note that it is not sufficient to take McDuff's original examples where $M = ST^*\Sigma$ is the unit cotangent bundle of a surface of genus at least 2. For in this case the contact structures one obtains on $M$ resp.\ $\overline{M}$ are isotopic to the canonical contact structure $\xi_{can}$ on $M$ and a horizontal contact structure $\xi_{LC}$ given by the Levi-Civita connection of a hyperbolic metric on $\Sigma$. The canonical contact structure always admits Weinstein, and hence Stein, fillings (\cite{CE}, Example 11.12) and by the classification of contact structures on $S^1$-bundles (see \cite{Gir2}, \cite{Ho}) the same is true of $\xi_{LC}$. Thus the resulting contact structure on the connected sum is also Stein fillable.

\section{Brieskorn Spheres and $PSL(2,\mathbb{R})$-structures}\label{main_result}
We consider the Brieskorn spheres $\overline{\Sigma}(2,3,6n+5)$ taken with the opposite orientation to that given by their description as the link of the complex singularity $z_1^{2} +z_2^3 + z_3^{6n+5} = 0$. These are Seifert fibred homology spheres, whose quotient orbifolds are hyperbolic for any natural number $n$ greater than one.

The manifold $\overline{\Sigma}(2,3,6n+5)$ admits a contact structure $\eta_{tan}$ that is tangential to the Seifert fibration (cf.\ \cite{Mas}, p.\ 1764). This contact structure has the property that it is isotopic to its conjugate $\overline{\eta}_{tan}$, which denotes the same contact structure taken with the opposite orientation, and we note this in the following proposition.
\begin{prop}\label{bundle_cover}
For any natural number $n \geq 1$ the manifold $\overline{\Sigma}(2,3,6n+5)$ admits a tangential contact structure $\eta_{tan}$. Moreover, any tangential contact structure is isotopic to its conjugate and is universally tight.
\end{prop}
\begin{proof}
We let $B$ denote the quotient orbifold of $\overline{\Sigma}(2,3,6n+5)$ given by the Seifert fibration. A tangential contact structure $\eta_{tan}$ induces a fibrewise cover $\overline{\Sigma}(2,3,6n+5) \to ST^*B$ to the unit cotangent bundle of the orbifold $B$ so that $\eta_{tan}$ is the pullback of the canonical contact structure $\xi_{can}$ on $ST^*B$ by (\cite{Mas}, Proposition 8.9). By assumption $B$ is a hyperbolic orbifold and hence $ST^*B$ is a compact quotient of $PSL(2,\mathbb{R})$ by a discrete lattice. Furthermore $\xi_{can}$ comes from a left-invariant contact structure on $PSL(2,\mathbb{R})$, which is the kernel of some left-invariant $1$-form, where we have identified $PSL(2,\mathbb{R})$ with $ST^*\mathbb{H}^2$ using the action of $PSL(2,\mathbb{R})$ on $\mathbb{H}^2$ via M\"obius transformations.

Using the notation of Example \ref{psl_ex}, any left-invariant $1$-form that is tangential lies in the span of $h^*$ and $l^*$, since $k$ generates the circle action. Moreover, since the linking form is negative definite on the span of $h^*$ and $l^*$, any non-zero form that is tangential determines a contact structure. Hence the space of tangential $PSL(2,\mathbb{R})$-invariant contact structures is connected, so in particular $\xi_{can}$ is isotopic to its conjugate and the same then holds for $\eta_{tan}$ by taking pullbacks. In general any tangential contact structure can be perturbed to a horizontal contact structure, which is then universally tight by (\cite{Mas}, Theorem A).
\end{proof}

\noindent Ghiggini has shown that $\overline{\Sigma}(2,3,6n+5)$ admits a contact structure which is strongly fillable, but admits no Stein fillings, when $n$ is even (\cite{Ghi}, Theorem 1.5). The only properties of the contact structures used in Ghiggini's proof of non-Stein fillability is that they are isotopic to their conjugates and that their $d_3$-invariant is $-\frac{3}{2}$. However, it follows from the classification of \cite{GhV} that all tight contact structures on $\overline{\Sigma}(2,3,6n+5)$ satisfy this latter constraint, thus in view of Proposition \ref{bundle_cover} we deduce the following.
\begin{thm}[\cite{Ghi}]\label{non_fill}
If $n$ is even, then a tangential contact structure $\eta_{tan}$ on $\overline{\Sigma}(2,3,6n+5)$ does not admit any Stein fillings.
\end{thm}
\begin{rem}
It is possible to deduce the fact that $d_3(\eta_{tan}) = -\frac{3}{2}$ without using the full force of the classification given in \cite{GhV}. One only needs the classification of contact structures with twisting number $-5$, for which the Heegaard-Floer computations used in \cite{Ghi} suffice, and the fact that at least one of these must be horizontal (\cite{Mas}, p.\ 1764).

One can further show that the contact structure $\eta_{tan}$ corresponds to $\eta_{n-1,0}$ in terms of the classification of tight contact structures on $\overline{\Sigma}(2,3,6n+5)$ given in \cite{GhV}. In this way, Theorem \ref{non_fill} also holds for $n$ odd and at least $2$ by (\cite{LiS}, Theorem 1.8).
\end{rem}
As pointed out to us by C. Wendl, the fact that $(\overline{\Sigma}(2,3,6n+5), \eta_{tan})$ can be realised as a boundary component of an exact filling with disconnected boundary shows that the notion of exact semi-fillability is strictly weaker than that of exact fillability, unlike for its weak and strong counterparts. We note this in the following proposition.
\begin{prop}\label{semi_false}
There exist infinitely many contact manifolds that are exactly semi-fillable but admit no exact fillings.
\end{prop}
\noindent With these preliminaries we may now construct examples of exactly fillable contact structures that admit no Stein fillings.
\begin{thm}\label{non_Stein_Exact}
There exist infinitely many exactly fillable contact manifolds that do not admit Stein fillings.
\end{thm}
\begin{proof}
We consider the fibrewise covering $\overline{\Sigma}(2,3,6n+5) \to ST^*B$ given in the proof of Proposition \ref{bundle_cover}. Since $ST^*B$ admits a $PSL(2,\mathbb{R})$-structure, the product $ST^*B \times [0,1]$ can be made into an exact symplectic filling with convex ends as in Example \ref{psl_ex} and by taking pullbacks the same is true of $\overline{\Sigma}(2,3,6n+5) \times [0,1]$. 

We let $\xi = \xi_0 \# \xi_1$ be the contact structure on $\overline{\Sigma}(2,3,6n+5) \# \Sigma(2,3,6n+5)$ given by attaching a symplectic 1-handle as in Proposition \ref{non_Stein_exact} and note that $\xi_0$ is tangential by construction. Then by Corollary \ref{connect_Stein_rem} we have that $\xi$ is Stein fillable if and only if $\xi_0$ and $\xi_1$ are Stein fillable. However, if $n$ is even the contact structure $\xi_0$ is not Stein fillable by Theorem \ref{non_fill} and it follows that $\xi$ is exactly fillable, but not Stein fillable.
\end{proof}
\noindent In fact, the argument used to show that $\eta_{tan}$ is not Stein fillable actually shows that it is not even exactly fillable (\cite{Ghi}, p.\ 1685). In view of the examples used in Theorem \ref{non_Stein_Exact} this then exhibits the failure of the analogue of Corollary \ref{connect_Stein_rem} for exactly fillable contact structures. 
\begin{cor}\label{connect_fail}
There exist infinitely many contact manifolds $(M_1,\xi_1),(M_2,\xi_2)$ such that the contact connected sum $(M_1\#M_2,\xi_1 \# \xi_2)$ is exactly fillable but $(M_1,\xi_1)$ is not.
\end{cor}


\noindent Since the contact structure $\eta_{tan}$ is isotopic to a deformation of a taut foliation, we further deduce the following.
\begin{cor}
There exist infinitely many contact structures that are deformations of taut foliations on homology spheres, which are not Stein fillable.
\end{cor}
The non-Stein fillable contact structure $\eta_{tan}$ is defined as the pullback of a tangential contact structure. This is completely analogous to the examples of Eliashberg in \cite{Eli3}, who showed that the pullbacks of the standard contact structure on $T^3 = ST^*T^2$ under suitable coverings admit weak, but not strong, symplectic fillings. This leads to the following question, a negative answer of which would provide a very large class of symplectically fillable contact structures without exact or Stein fillings.
\begin{qu}
Let $M$ be a non-trivial fibrewise cover of $ST^*\Sigma$, where $\Sigma$ is a closed, hyperbolic surface. Is the pullback of the canonical contact structure exactly or even Stein fillable?
\end{qu}
One may also formulate this question for the cotangent bundle of any hyperbolic orbifold. However, here the situation appears to be more subtle, since for example $\overline{\Sigma}(2,3,11)$ carries a unique tight contact structure, which is both Stein fillable and a fibrewise pullback of the canonical contact structure under a non-trivial covering map. A natural way of excluding this example would be to assume that the twisting number of the contact structure on the finite cover is not minimal amongst those contact structures that are isotopic to horizontal ones. This would also fit in with similar phenomena that occur in the classification of horizontal contact structures on $S^1$-bundles (cf.\ \cite{Gir2}, \cite{Ho}).

We finally remark that in the case of coverings of $ST^*T^2$ the obstruction to the existence of a strong filling can be seen as given by the Giroux torsion of the contact structure (cf.\ \cite{Gay}). Thus one might hope that there is some similar type of obstruction for covers of the unit cotangent bundle of a higher genus surface or even on more general Seifert fibred spaces.




\end{document}